\documentclass[reqno,11pt]{article}

\usepackage[a4paper,left=27mm,right=27mm,top=27mm,bottom=27mm]{geometry}
\usepackage{amsmath,amssymb,amsfonts,amsthm,mathtools}
\usepackage[svgnames]{xcolor}
\usepackage[colorlinks=true,linkcolor=DarkRed,citecolor=NavyBlue,
            urlcolor=NavyBlue]{hyperref}

\numberwithin{equation}{section}

\theoremstyle{definition}
\newtheorem{definition}{Definition}[section]
\newtheorem{problem}[definition]{Problem}
\theoremstyle{plain}
\newtheorem{proposition}[definition]{Proposition}
\newtheorem{corollary}[definition]{Corollary}
\newtheorem{conj}[definition]{Conjecture}
\theoremstyle{remark}
\newtheorem{remark}[definition]{Remark}
\usepackage{graphics,graphicx}

\graphicspath{{./pics/}}

\newcommand{\Bcal}{\mathcal B}
\newcommand{\Mcal}{\mathcal M}
\newcommand{\Haus}{\mathcal H}
\newcommand{\Amin}{A_{\min}}

\title{Dangling points in area-minimizing Meissner polyhedra}
\author{Beniamin Bogosel}
\date{\today}

\begin{document}
\maketitle

\begin{abstract}
Meissner polyhedra are constant-width bodies obtained from extremal finite
sets of unit diameter.  Such a generating set may contain dangling points,
namely points having exactly two diametric neighbors.  This article studies
whether these points can play an essential role in surface-area
minimization.  Given an extremal set we show that the
smallest surface area among the Meissner polyhedra based on it cannot increase by deleting a dangling point. Adding a dangling point cannot decrease the smallest achievable surface area. 
This reduces the search for area-minimizing Meissner polyhedra to
generating sets without dangling points.
\end{abstract}



\section{Introduction}

A convex body $K\subset\mathbb R^3$ has \emph{constant width one} if the
distance between any two parallel supporting planes of $K$ is equal to one.
The three-dimensional Blaschke--Lebesgue problem asks for the least possible
volume of such a body.  The classical conjecture, generally attributed to
Bonnesen and Fenchel, asserts that the minimum is attained by the Meissner
tetrahedra; see
\cite{Bonnesen-Fenchel,Meissner-Schilling,kawohl-webe}.  This remains one of
the central open problems concerning bodies of constant width.  Numerical
shape optimization gives additional evidence for the conjectured
minimizers; see, for example, \cite{AntunesBogosel22}.

For a body $K\subset\mathbb R^3$ of constant width one, Blaschke's relation
links its volume and surface area by \cite{Bonnesen-Fenchel,hynd-vol-per}
\begin{equation}\label{eq:blaschke-relation}
 |K|=\frac12\Haus^2(\partial K)-\frac{\pi}{3}.
\end{equation}
Consequently, volume minimization and surface-area minimization are
equivalent in this class.  Meissner polyhedra form a particularly important
finite-dimensional family of constant-width bodies.  Their construction was
developed in \cite{montejano} and revisited in \cite{meissner_hynd}, where
their density in the class of constant-width bodies was established.  Their
surface area and volume admit explicit expressions in terms of pairs of dual
edges; see \cite{bogosel_Meissner,hynd-vol-per}.  This converts part of the
three-dimensional Blaschke--Lebesgue problem into a collection of finite
geometric optimization problems.

We recall the finite configurations underlying this construction.  Let
$X\subset\mathbb R^3$ be a finite set with $m\geq4$ elements and define
\begin{equation}\label{eq:diameter-one}
 \operatorname{diam}X=\max\{|p-q|:p,q\in X\}.
\end{equation}
A pair $\{p,q\}\subset X$ is called \emph{diametric} if
$|p-q|=\operatorname{diam}X$.  After a scaling, we shall always assume that
$\operatorname{diam}X=1$.  Denote the number of diametric pairs by
\begin{equation}\label{eq:diametric-count}
 e(X)=\#\bigl\{\{p,q\}\subset X:|p-q|=1\bigr\}.
\end{equation}
The solution of the V\'{a}zsonyi problem in three dimensions states that
$e(X)\leq2m-2$.  The equality case has a rich ball-polyhedral structure;
see \cite{disk-polygons,meissner_hynd}.

\begin{definition}[extremal finite set]\label{def:extremal-set}
 A finite set $X\subset\mathbb R^3$ with $m\geq4$ elements is an
 \emph{extremal finite set of unit diameter} if
 \begin{equation}\label{eq:extremal-count}
  \operatorname{diam}X=1
  \qquad\text{and}\qquad
  e(X)=2m-2.
 \end{equation}
 Equivalently, $X$ realizes the largest possible number of diametric pairs
 among $m$-point subsets of $\mathbb R^3$ of diameter one.
\end{definition}

For any nonempty set $S\subset\mathbb R^3$, put
\begin{equation}\label{eq:ball-polyhedron}
 \Bcal(S)=\bigcap_{p\in S}\overline B(p,1).
\end{equation}
For finite $X$, this is the ball polyhedron associated with $X$.
When $X$ is extremal, its points are precisely the vertices of
$\Bcal(X)$, with vertices and faces related by the standard ball-polyhedral
duality.  In particular, the edges of $\Bcal(X)$ are grouped into $m-1$
dual pairs
\begin{equation}\label{eq:dual-pairs}
 (e_1,e_1'),\ldots,(e_{m-1},e_{m-1}').
\end{equation}
Set $e_i^0=e_i$ and $e_i^1=e_i'$.  For
$\varepsilon=(\varepsilon_1,\ldots,\varepsilon_{m-1})
\in\{0,1\}^{m-1}$, define
\begin{equation}\label{eq:meissner-choice}
 M_\varepsilon(X)=
 \Bcal\left(X\cup\bigcup_{i=1}^{m-1}e_i^{\varepsilon_i}\right).
\end{equation}
Choosing the centers on one edge in each dual pair performs the corresponding
Meissner smoothing on its dual edge.  Every body
$M_\varepsilon(X)$ has constant width one.  We denote by $\Mcal(X)$ the
finite family of all such Meissner polyhedra; different binary choices need
not always give distinct bodies.  The relevant structural facts and the
constant-width property are proved in
\cite{montejano,meissner_hynd}; the dependence of surface area on the dual
edge pairs is described in \cite{bogosel_Meissner}.

Not every vertex of an extremal set has the same local role.  For
$x\in X$, define its diametric valence by
\begin{equation}\label{eq:diametric-valence}
 \operatorname{val}_X(x)
 =\#\{p\in X\setminus\{x\}:|p-x|=1\}.
\end{equation}

\begin{definition}[dangling point]\label{def:dangling-point}
 Let $X$ be an extremal finite set of unit diameter.  A point $x\in X$ is
 called a \emph{dangling point} if
 \begin{equation}\label{eq:dangling-valence}
  \operatorname{val}_X(x)=2.
 \end{equation}
 In the ball polyhedron $\Bcal(X)$, such a point belongs to exactly two
 faces.  After $x$ is deleted from the generating set, it lies in the
 relative interior of a circular edge of $\Bcal(X\setminus\{x\})$; see
 \cite[Section~3]{meissner_hynd} and Figure~\ref{fig:dangling-diam}.
\end{definition}

\begin{figure}
	\centering
	\includegraphics[width=0.35\textwidth]{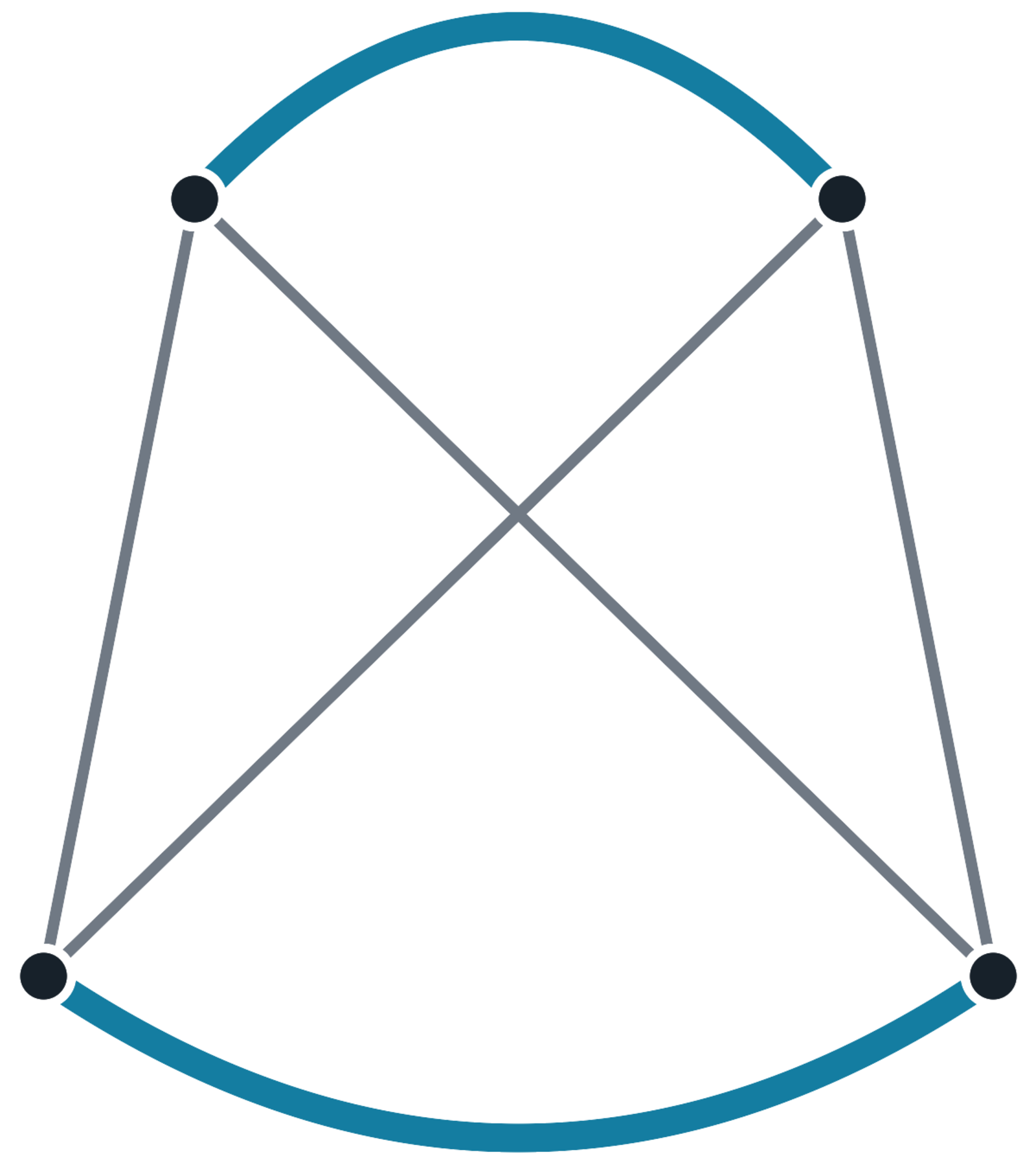} \qquad 
	\includegraphics[width=0.35\textwidth]{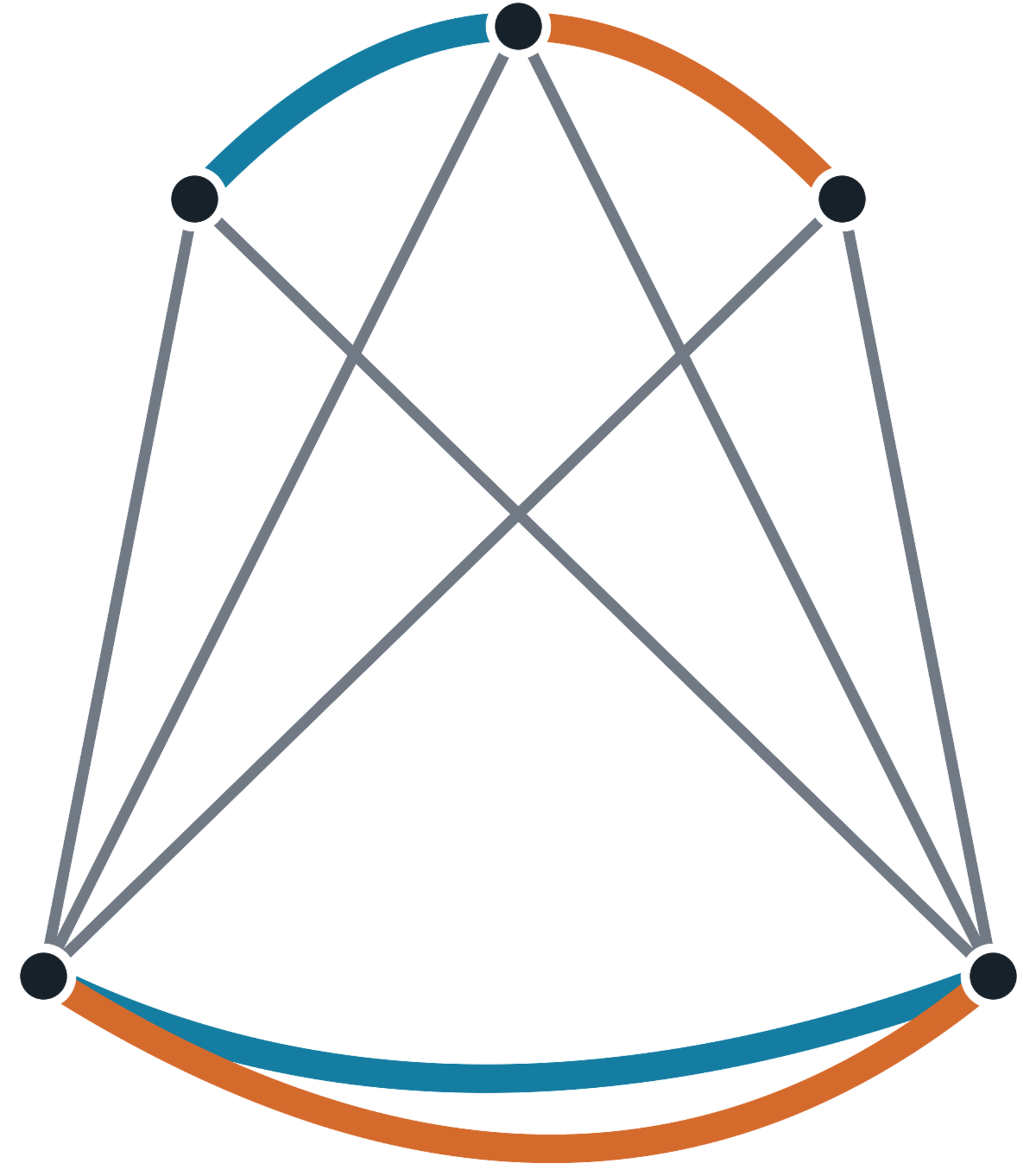}
	\caption{A generic dual edge pair, with the four diameters joining its
	endpoints (left).  After a dangling point is inserted on one edge, the
	original pair is replaced by two dual pairs (right); arcs with the same
	color are dual.}
	\label{fig:dangling-diam}
\end{figure}

If $x$ is dangling, deleting it removes exactly two diametric pairs.  Hence
$X\setminus\{x\}$ has $m-1$ points and
\begin{equation}\label{eq:deletion-extremal-count}
 e(X\setminus\{x\})=(2m-2)-2=2(m-1)-2.
\end{equation}
Thus deletion preserves extremality.  Geometrically, the dangling point
subdivides one circular edge of the smaller ball polyhedron, and the local
change is accompanied by the corresponding subdivision in the dual-edge
description.  This makes dangling deletion a natural elementary operation
on generating sets.

For an extremal set $X$, define
\begin{equation}\label{eq:Amin-definition}
 \Amin(X)
 =\min_{M\in\Mcal(X)}\Haus^2(\partial M).
\end{equation}
The minimum exists because $\Mcal(X)$ is obtained from finitely many
smoothing choices.  Thus $\Amin(X)$ allows the smoothing to be chosen
optimally for the fixed generating set $X$; it is not the area of an
arbitrarily prescribed smoothing.

The guiding principle of this article is that dangling points should play no
essential role in the minimization of area.  More precisely, we address the
following problem.

\begin{problem}[dangling deletion problem]\label{prob:dangling-deletion}
 Let $X_m\subset\mathbb R^3$ be an extremal finite set of unit diameter with
 $m\geq5$ points, and let $x\in X_m$ be a dangling point.  Prove that
 \begin{equation}\label{eq:main-target}
  \boxed{
  \Amin(X_m)\geq
  \Amin\bigl(X_m\setminus\{x\}\bigr).}
 \end{equation}
\end{problem}

The inequality allows dangling points to be removed
successively without increasing the least achievable area.  The search for
an area-minimizing Meissner polyhedron could therefore be reduced to a
dangling-free core of an extremal set, obtained after all dangling points
have been deleted.  By \eqref{eq:blaschke-relation}, the identical reduction
would hold for volume.

It is important to distinguish this reduction from a strict exclusion
statement.  Inequality \eqref{eq:main-target} shows that a dangling point is
unnecessary for attaining a smaller area.  Equality may nevertheless occur,
for example when insertion of the point only subdivides a region on which
the same smoothing is already performed. 

\section{Area comparison under deletion of one dangling point}
\label{sec:dangling-area-comparison}

We now make the comparison in Problem~\ref{prob:dangling-deletion} for one
dangling point.  Let $X=X_m$ be an extremal set, let $x\in X$ be
dangling, and put $Y=X\setminus\{x\}$.  Write $a,b\in Y$ for the two
diametric neighbours of $x$.  Deleting $x$ removes exactly two diametric
pairs, so $Y$ is again extremal.  By the edge--face duality of extremal
ball polyhedra \cite[Section~3]{meissner_hynd}, there are vertices
$y,z\in Y$ such that $x$ lies in the relative interior of the circular
edge $e_{yz}$ of $\Bcal(Y)$.  This edge is dual to an edge $e_{ab}$.
In particular,
\[
 |a-y|=|a-z|=|b-y|=|b-z|=1.
\]
Indeed, the constraints from $a$ and $b$ are the only unit-distance
constraints active at $x$; all the other constraints defining
$\Bcal(Y)$ are strict there.  Thus $x$ is a relative interior point of
the circle arc cut out by $\partial B(a,1)\cap\partial B(b,1)$.
Passing from $Y$ to $X$ replaces the single dual pair
\[
 (e_{yz},e_{ab})
\]
by the two dual pairs
\[
 (e_{yx},e_{ab}^{\,1}),
 \qquad
 (e_{xz},e_{ab}^{\,2}).
\]
The two edges $e_{ab}^{\,1}$ and $e_{ab}^{\,2}$ are the two sides of
the spherical digon dual to $x$.  They have the same endpoints and the
same angular length, although they need not be the same circular arc.  Every
other dual pair is unchanged.  Indeed, in the strongly self-dual face
complex, the valence-two vertex $x$ is dual to this digon.  Deleting $x$
suppresses the subdivision at $x$ and replaces the two sides of the digon
by the single edge $e_{ab}$, without changing any other incidence; see
\cite[Section~3, in particular Theorem~3.6]{meissner_hynd}.  See
Figure~\ref{fig:dangling-options}.

\begin{figure}
	\centering
	\begin{tabular}{cccc}
		\includegraphics[width=0.22\textwidth]{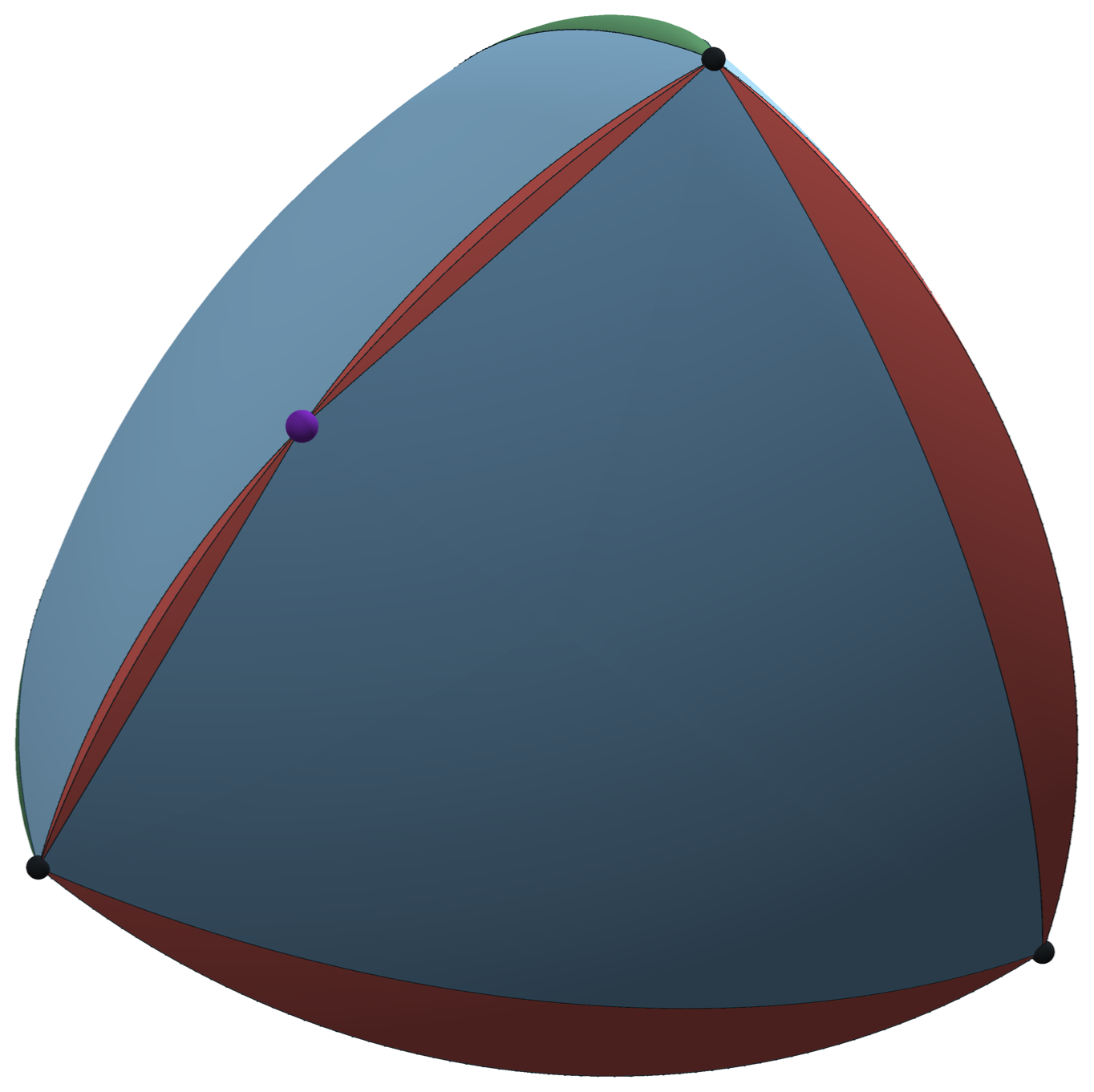} & 
		\includegraphics[width=0.22\textwidth]{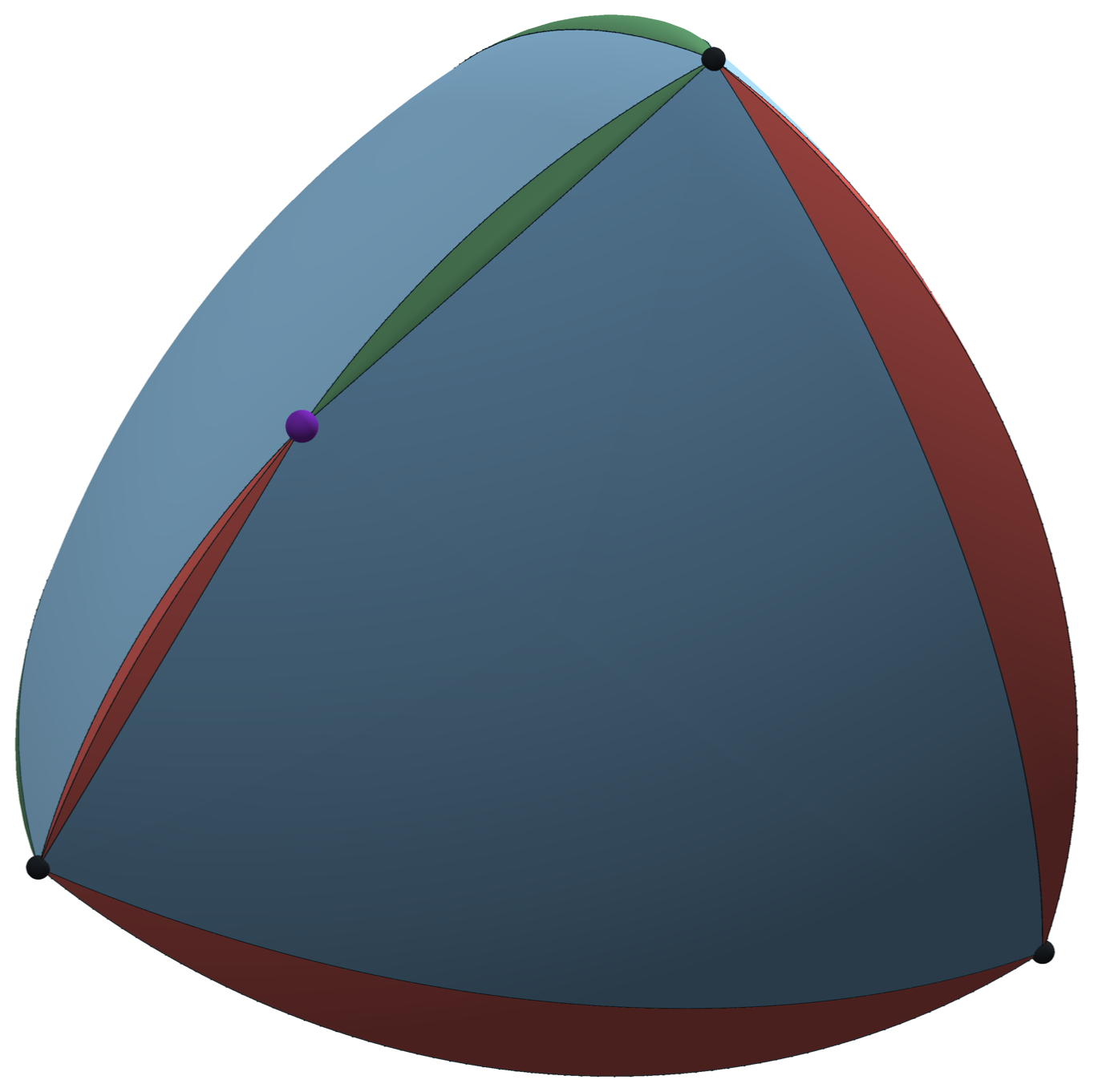} &
		\includegraphics[width=0.22\textwidth]{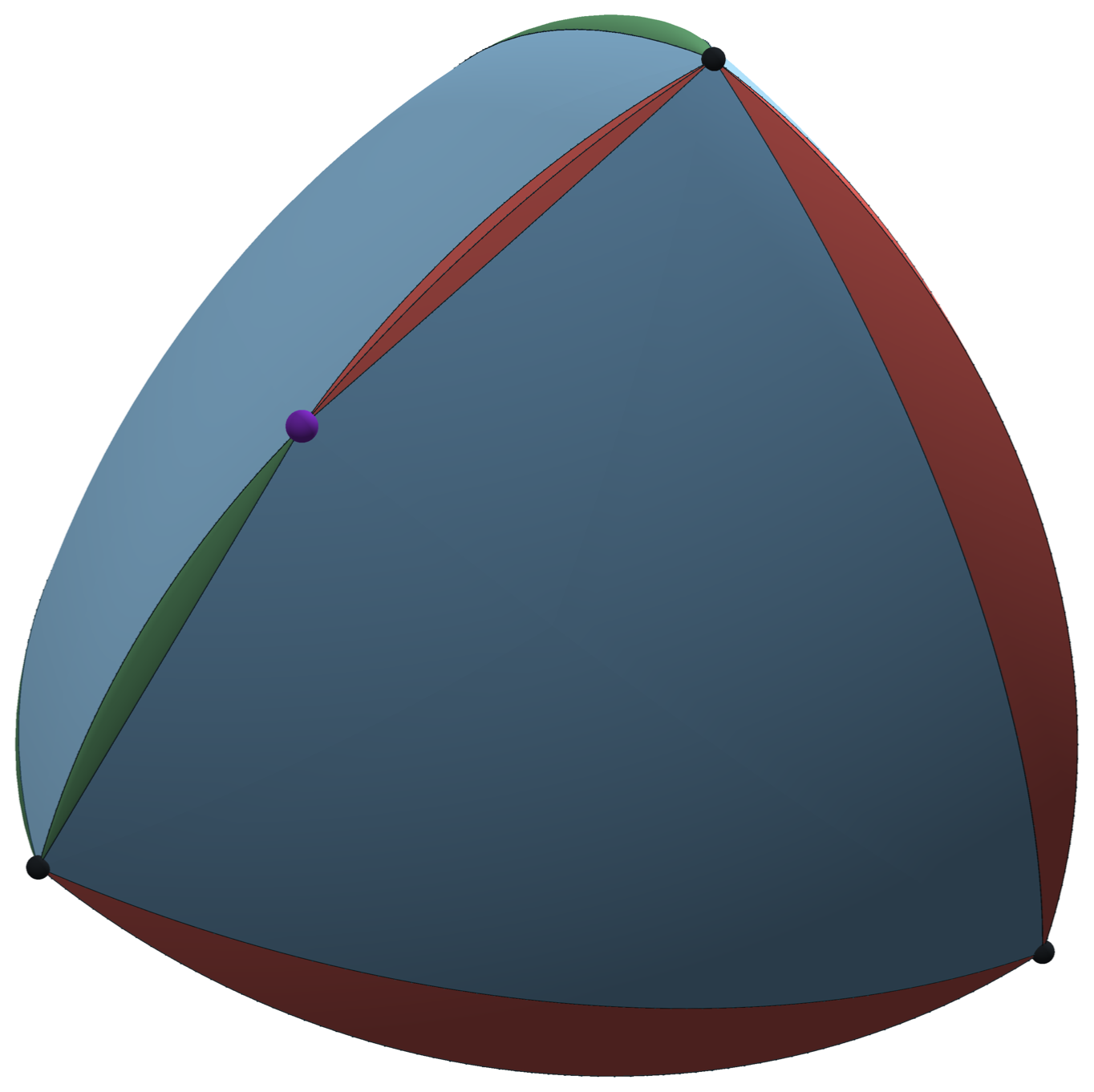} & 
		\includegraphics[width=0.22\textwidth]{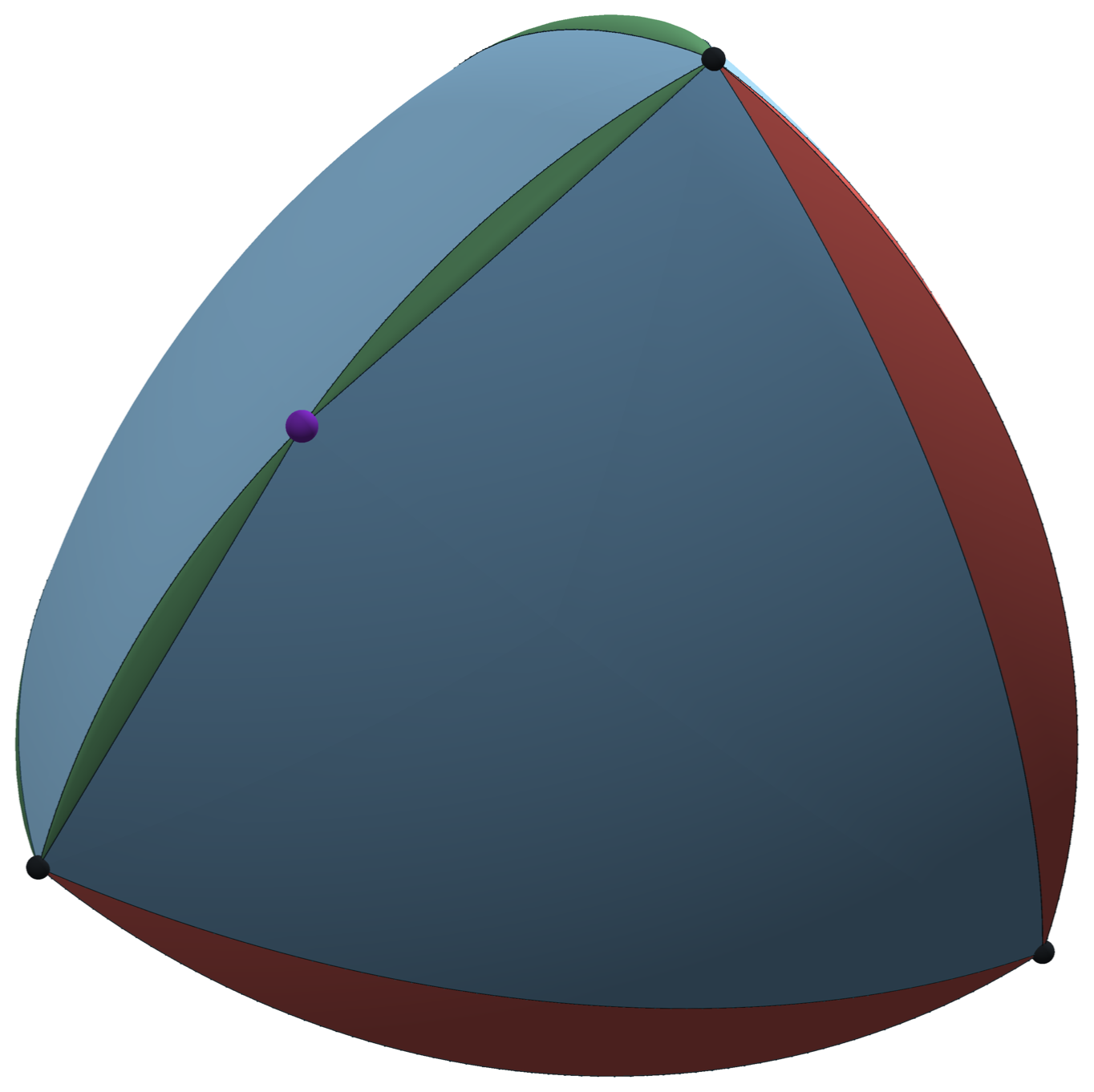} \\
		\includegraphics[width=0.22\textwidth]{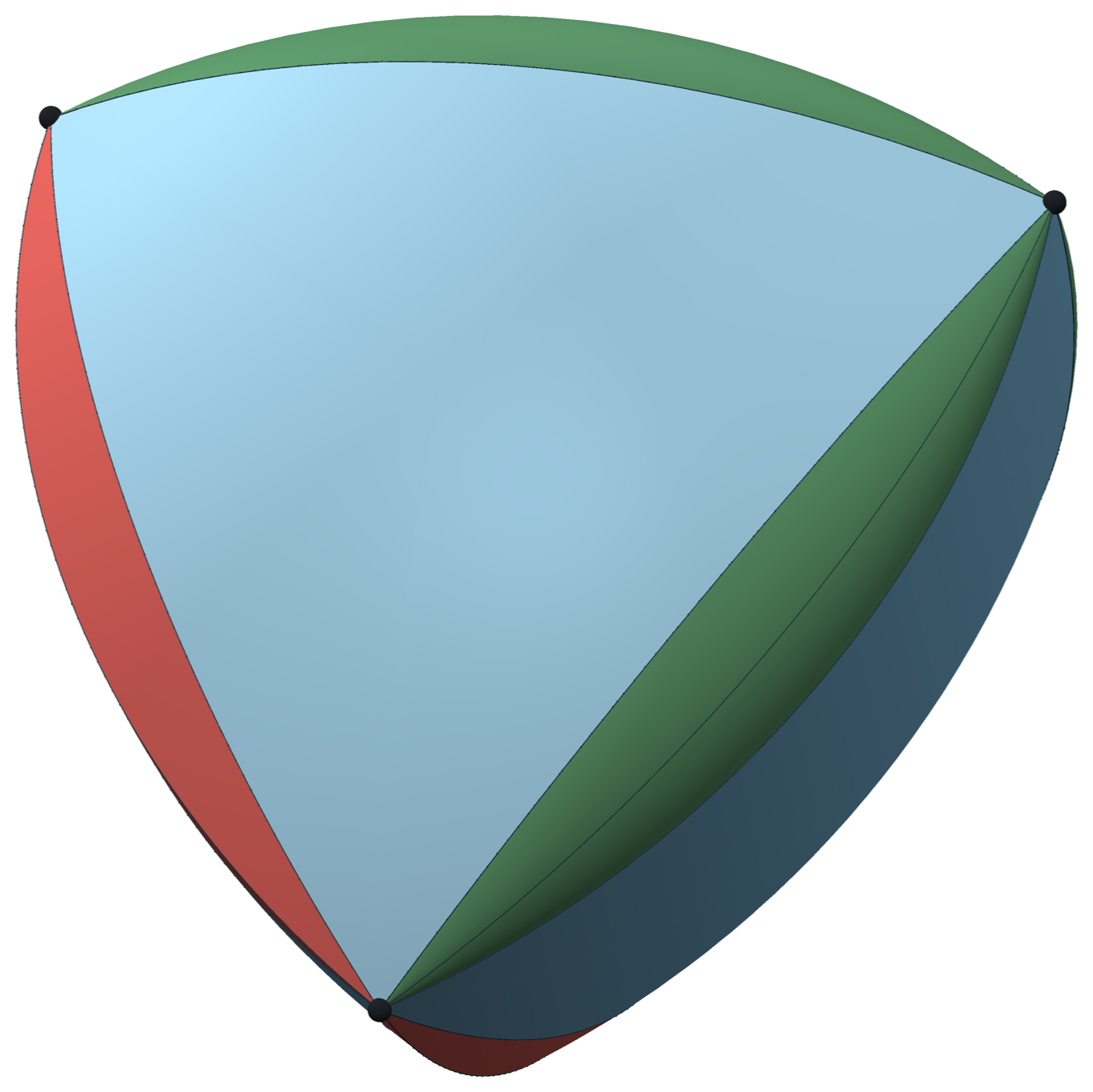} & 
		\includegraphics[width=0.22\textwidth]{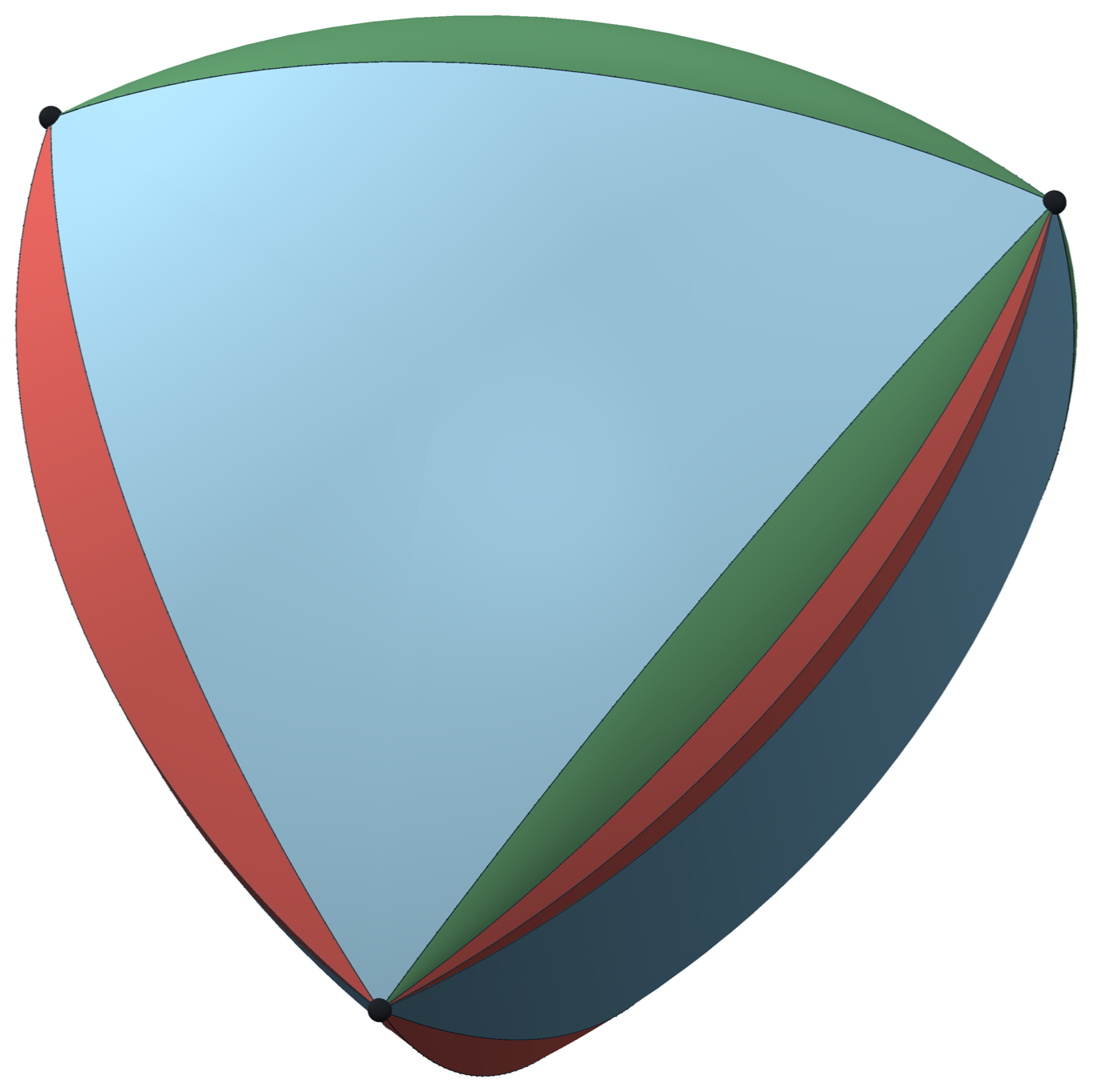} &
		\includegraphics[width=0.22\textwidth]{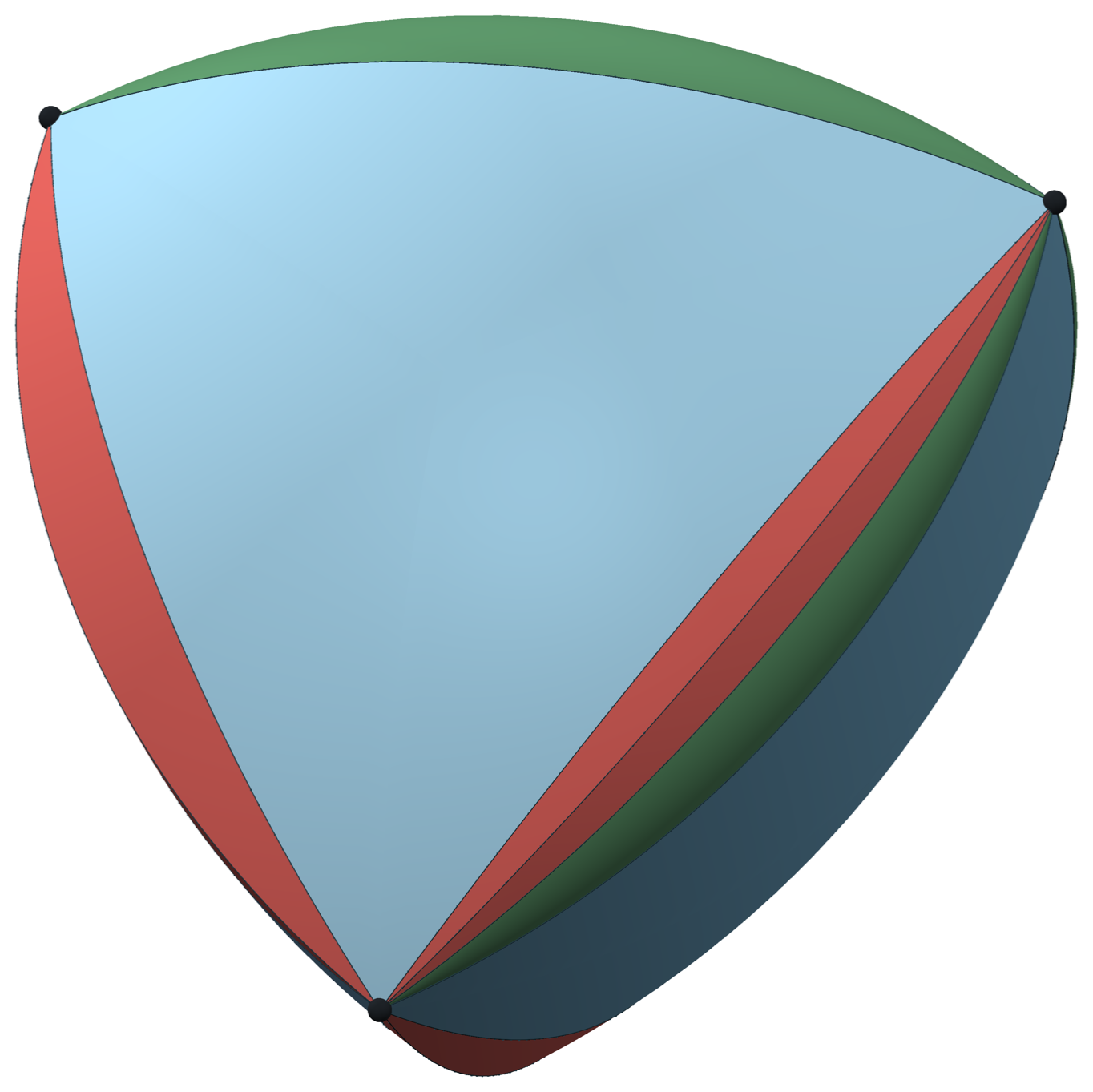} & 
		\includegraphics[width=0.22\textwidth]{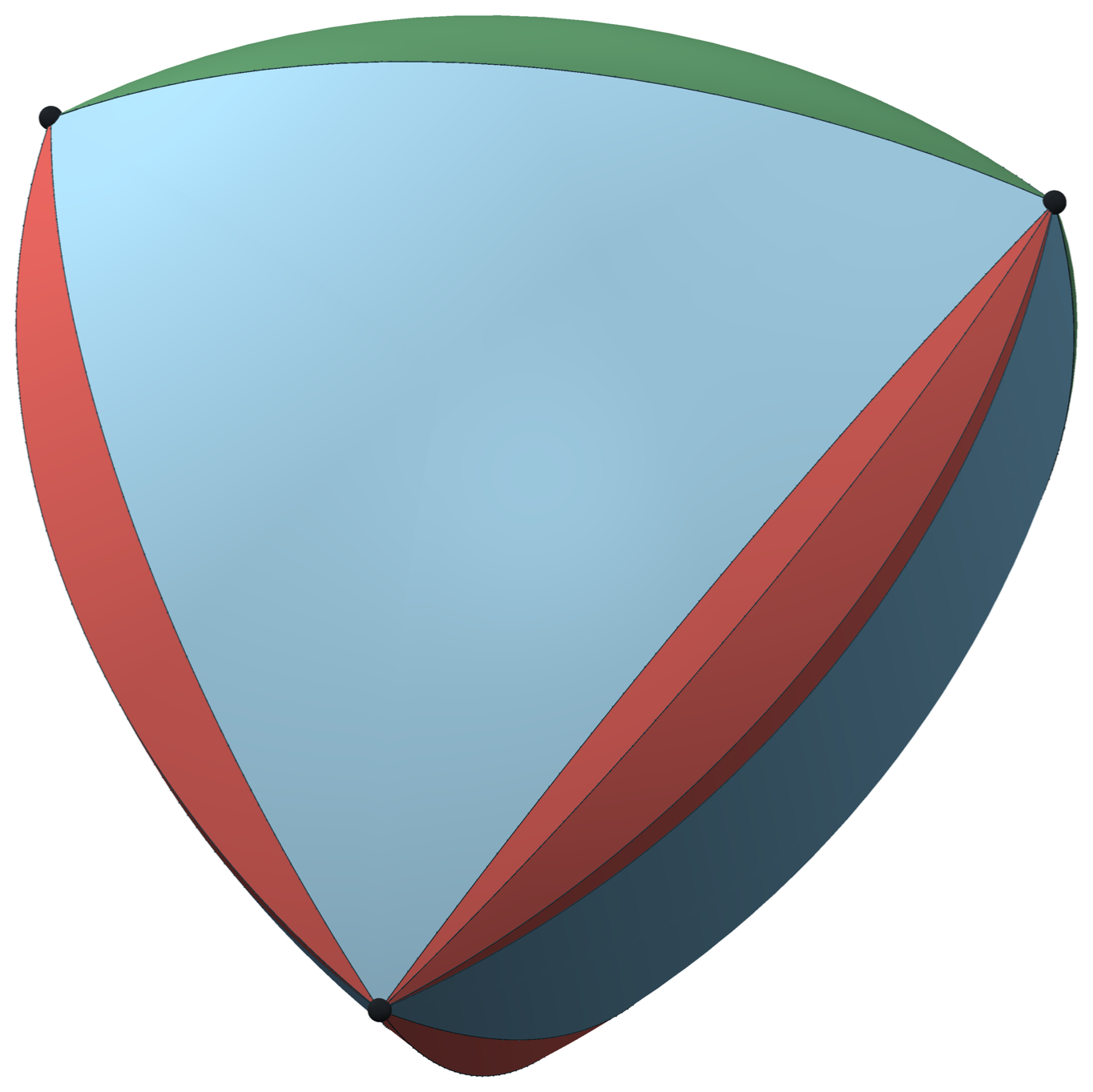} \\
	\end{tabular}
	\caption{A dangling point and its associated four smoothing options for the two pairs of dual edges. Bottom row shows dual edges compared to the first one}
	\label{fig:dangling-options}
\end{figure}

For an edge with endpoints $u,v$, define the spherical length of a geodesic joining $u,v$ on a unit sphere by
\begin{equation}\label{eq:draft-angular-length}
 \theta(uv)=2\arcsin\frac{|u-v|}{2}\in[0,\pi/3].
\end{equation}
Set
\begin{equation}\label{eq:draft-local-angles}
 \alpha=\theta(ab),\qquad
 \beta=\theta(yz),\qquad
 \beta_1=\theta(yx),\qquad
 \beta_2=\theta(xz).
\end{equation}
All these angles are positive and at most $\pi/3$; in particular, no
degenerate edge occurs in the comparisons below.
The surface-area formula for Meissner polyhedra derived in \cite{bogosel_Meissner} is
\begin{equation}\label{eq:draft-bogosel-area}
 \Haus^2(\partial M)
 =2\pi-2\sum_i f\bigl(\theta(e_i),\theta(e_i')\bigr),
\end{equation}
where $e_i$ is chosen as a set of centres, $e_i'$ is consequently
smoothed, and
\begin{equation}\label{eq:draft-f-definition}
 f(s,t)
 =t\cos\frac t2\,
  \arcsin\left(\frac{\sin(s/2)}{\cos(t/2)}\right).
\end{equation}
Since the area is $2\pi$ minus twice the sum in
\eqref{eq:draft-bogosel-area}, minimizing area means maximizing each local
summand.  The comparison proved in \cite[Appendix~A]{bogosel_Meissner} is
\begin{equation}\label{eq:draft-longest-smoothing}
 0\leq s\leq t\leq\frac\pi3
 \quad\Longrightarrow\quad
 f(s,t)\geq f(t,s),
\end{equation}
with strict inequality when $0<s<t$.  Thus the area-minimizing rule is:
choose centres on the shorter edge and smooth the longer edge.  It is useful
to write the optimally chosen local summand as
\begin{equation}\label{eq:draft-H-definition}
 H(s,t)=\max\{f(s,t),f(t,s)\}
 =f\bigl(\min\{s,t\},\max\{s,t\}\bigr).
\end{equation}

Let $Q$ be the sum of the optimally chosen $H$-summands over all dual
pairs unchanged by deleting $x$.  Formula
\eqref{eq:draft-bogosel-area} gives the two exact values
\begin{align}
 \Amin(Y)
 &=2\pi-2\bigl[Q+H(\beta,\alpha)\bigr],
 \label{eq:draft-area-Y}\\
 \Amin(X)
 &=2\pi-2\bigl[Q+H(\beta_1,\alpha)
                         +H(\beta_2,\alpha)\bigr].
 \label{eq:draft-area-X}
\end{align}
We distinguish $\beta\leq\alpha$ from $\beta>\alpha$; the smoothing
ambiguity at equality is recorded separately.

The circle $\partial B(a,1)\cap\partial B(b,1)$ containing
$e_{yx}\cup e_{xz}=e_{yz}$ has radius
\begin{equation}\label{eq:draft-circle-radius}
 r=\cos\frac\alpha2.
\end{equation}
Let $t_1,t_2>0$ be the central angles of $e_{yx},e_{xz}$, and put
$T=t_1+t_2$.  By the standard edge description for extremal ball
polyhedra, these edges are minor circular arcs
\cite[proof of Lemma~4.2]{meissner_hynd}; hence
$0<t_i<T<\pi$.  Then
\begin{equation}\label{eq:draft-psi}
 \beta_i=\psi(t_i),\qquad
 \beta=\psi(T),\qquad
 \psi(t)=2\arcsin\left(r\sin\frac t2\right).
\end{equation}
In particular, $\psi$ is strictly increasing on the relevant range.
Moreover,
\begin{equation}\label{eq:draft-additivity}
 f(\beta_1,\alpha)+f(\beta_2,\alpha)=f(\beta,\alpha).
\end{equation}
Indeed, \eqref{eq:draft-f-definition} and
\eqref{eq:draft-psi} give
\begin{equation}\label{eq:draft-linear-branch}
 f(\psi(t),\alpha)
 =\alpha r\,
  \arcsin\left(
   \frac{\sin(\psi(t)/2)}{r}
 \right)
 =\frac{\alpha r t}{2},
\end{equation}
where the last equality uses $0<t<T<\pi$, so that
$\arcsin(\sin(t/2))=t/2$, and $t_1+t_2=T$.

\begin{proposition}[The original edge is no longer than its dual]
\label{prop:draft-beta-at-most-alpha}
If $\beta\leq\alpha$, then
\begin{equation}\label{eq:draft-case-short-value}
 \Amin(X)=\Amin(Y)
 =2\pi-2\bigl[Q+f(\beta,\alpha)\bigr].
\end{equation}
Compatible minimizing choices produce the same Meissner body.
\end{proposition}

\begin{proof}
Since $0<t_i<T$ and $\psi$ is strictly increasing,
\[
 \beta_1<\beta\leq\alpha,
 \qquad
 \beta_2<\beta\leq\alpha.
\]
The longest-edge rule \eqref{eq:draft-longest-smoothing} therefore gives
\[
 H(\beta,\alpha)=f(\beta,\alpha),
 \qquad
 H(\beta_i,\alpha)=f(\beta_i,\alpha),\quad i=1,2.
\]
Substitution in \eqref{eq:draft-area-Y}--\eqref{eq:draft-area-X}, followed
by the exact additivity \eqref{eq:draft-additivity}, yields
\[
 \Amin(X)-\Amin(Y)
 =2\bigl[f(\beta,\alpha)
          -f(\beta_1,\alpha)-f(\beta_2,\alpha)\bigr]
 =0,
\]
which is precisely \eqref{eq:draft-case-short-value}.

There is also equality at the level of bodies.  Choose the edge
$e_{yz}$ as a set of centres for $Y$, thereby smoothing its dual
$e_{ab}$.  For $X$, the two shorter edges $e_{yx}$ and $e_{xz}$ are
chosen as centres.  Since
\[
 Y\cup e_{yz}=X\cup e_{yx}\cup e_{xz},
\]
where $x\in e_{yz}$, the complete sets of centres in the two intersections
of unit balls are identical after making the same choices on every
unchanged dual pair.  Hence the resulting Meissner bodies coincide.
\end{proof}

\begin{remark}
When $\beta=\alpha$, the two orientations of the unsplit pair
$(e_{yz},e_{ab})$ give the same minimal area for $Y$.  The subdivided pairs
of $X$ still have unique minimizing orientations because
$\beta_i<\beta=\alpha$.  Choosing $e_{yz}$ as the set of centres for $Y$
is compatible with these orientations and gives the same body as for $X$.
The opposite orientation for $Y$ may give a different body with the same
area.
\end{remark}

\begin{proposition}[The original edge is longer than its dual]
\label{prop:draft-beta-greater-alpha}
If $\beta>\alpha$, then deletion of the dangling point strictly decreases
the optimally smoothed area:
\begin{equation}\label{eq:draft-case-long-strict}
 \Amin(Y)<\Amin(X).
\end{equation}
More explicitly,
\begin{align}
 \Amin(Y)
 &=2\pi-2\bigl[Q+f(\alpha,\beta)\bigr],
 \label{eq:draft-case-long-Y}\\
 \Amin(X)-\Amin(Y)
 &=2\Bigl[f(\alpha,\beta)
       -H(\beta_1,\alpha)-H(\beta_2,\alpha)\Bigr]>0.
 \label{eq:draft-case-long-difference}
\end{align}
\end{proposition}

\begin{proof}
The identity \eqref{eq:draft-case-long-Y} follows immediately from
$\beta>\alpha$, the longest-edge rule, and
\eqref{eq:draft-area-Y}.  We prove the strict inequality in
\eqref{eq:draft-case-long-difference}.

There are three exhaustive cases.

\smallskip
\noindent\emph{(i) Both subedges satisfy $\beta_i\leq\alpha$.}
The longest-edge rule and \eqref{eq:draft-additivity} give
\[
 H(\beta_1,\alpha)+H(\beta_2,\alpha)
 =f(\beta_1,\alpha)+f(\beta_2,\alpha)
 =f(\beta,\alpha).
\]
But $\beta>\alpha$, so strictness in
\eqref{eq:draft-longest-smoothing} gives
\[
 H(\beta,\alpha)=f(\alpha,\beta)>f(\beta,\alpha).
\]

\smallskip
\noindent\emph{(ii) Exactly one subedge is longer than $e_{ab}$.}
Assume without loss of generality that
$\beta_1\leq\alpha<\beta_2$.  Define
\begin{equation}\label{eq:draft-switching-gain}
 D_\alpha(t)=f(\alpha,\psi(t))-f(\psi(t),\alpha).
\end{equation}
This is the gain in the maximized local summand obtained by switching the
smoothing choice for edges with lengths $\psi(t)\geq\alpha$; the corresponding
decrease in area is $2D_\alpha(t)$.  We first note that $D_\alpha$ is strictly
increasing on this range.  If
$0<s<t\leq\pi/3$, direct differentiation of
\eqref{eq:draft-f-definition} gives
\begin{equation}\label{eq:draft-mixed-derivative-difference}
 \frac{\partial^2 f}{\partial s\,\partial t}(s,t)
 -\frac{\partial^2 f}{\partial s\,\partial t}(t,s)
 =\frac{\sin(s/2)\sin(t/2)}
 {8\bigl(1-\sin^2(s/2)-\sin^2(t/2)\bigr)^{3/2}}
   \bigl(t\sin s-s\sin t\bigr)>0.
\end{equation}
The last inequality follows because $u\mapsto\sin u/u$ is strictly
decreasing on $(0,\pi)$.  The derivative with respect to $s$ of
\[
 \frac{\partial f}{\partial t}(s,t)
 -\frac{\partial f}{\partial s}(t,s)
\]
is the left-hand side of
\eqref{eq:draft-mixed-derivative-difference}, while
\[
 \frac{\partial f}{\partial t}(0,t)
 =\frac{\partial f}{\partial s}(t,0)=0.
\]
Integration therefore gives
\begin{equation}\label{eq:draft-derivative-comparison}
 \frac{\partial f}{\partial t}(s,t)
 >
 \frac{\partial f}{\partial s}(t,s),
 \qquad 0<s<t\leq\frac\pi3.
\end{equation}
Since $\psi'(t)>0$, differentiation of
\eqref{eq:draft-switching-gain} and
\eqref{eq:draft-derivative-comparison} gives
\begin{equation}\label{eq:draft-D-increasing}
 D_\alpha'(t)
 =\left[
   \frac{\partial f}{\partial t}(\alpha,\psi(t))
   -\frac{\partial f}{\partial s}(\psi(t),\alpha)
  \right]\psi'(t)>0
\end{equation}
whenever $\psi(t)>\alpha$.  Also $D_\alpha(t)=0$ when
$\psi(t)=\alpha$.

The longest-edge rule and additivity now give
\begin{align*}
 H(\beta_1,\alpha)+H(\beta_2,\alpha)
 &=f(\beta,\alpha)+D_\alpha(t_2),\\
 H(\beta,\alpha)&=f(\beta,\alpha)+D_\alpha(T).
\end{align*}
Since $x$ is an interior point, $t_2<T$, and hence
\[
 f(\beta,\alpha)+D_\alpha(t_2)
 <f(\beta,\alpha)+D_\alpha(T).
\]
The same argument covers the symmetric alternative.

\smallskip
\noindent\emph{(iii) Both subedges satisfy $\beta_i>\alpha$.}
We continue to use $D_\alpha$ and its strict monotonicity established in
case~(ii).  On the active range $\psi(t)\geq\alpha$, define
\begin{equation}\label{eq:draft-galpha}
	g_\alpha(t)=f(\alpha,\psi(t)).
\end{equation}
We claim that $g_\alpha$ is convex on this range.

Put
\[
v=\frac{\psi(t)}2,\qquad
q=\tan v,\qquad
z=\frac{\sin(\alpha/2)}{\cos v}.
\]
Since $\psi(t)\geq\alpha$ and $\psi(t)\leq\pi/3$, we have
\begin{equation}\label{eq:draft-zq-range}
	0<z\leq q\leq\frac1{\sqrt3}.
\end{equation}
Indeed, $v\geq\alpha/2$ gives
$\sin(\alpha/2)\leq\sin v$, and hence $z\leq\tan v=q$.

By \eqref{eq:draft-psi},
\[
\sin v=\cos\frac{\alpha}{2}\sin\frac t2,
\]
and therefore
\begin{equation}\label{eq:draft-v-derivative}
	\frac{dv}{dt}=\frac12\sqrt{1-z^2}.
\end{equation}
Moreover, from the definition of $f$,
\[
g_\alpha(t)
=2v\cos v\,\arcsin z.
\]
A direct differentiation, using \eqref{eq:draft-v-derivative}, gives
\begin{equation}\label{eq:draft-g-second-derivative}
	g_\alpha''(t)
	=\frac12\cos v\,(v+2q)\,h_v(z),
\end{equation}
where
\begin{equation}\label{eq:draft-h-function}
	h_v(z)
	=z\sqrt{1-z^2}+(B_vz^2-1)\arcsin z,
	\qquad
	B_v=\frac{q+(1+q^2)v}{v+2q}.
\end{equation}
Thus it remains to prove that $h_v(z)\geq0$ on the range
\eqref{eq:draft-zq-range}.

For fixed $v$, one has $h_v(0)=0$, while differentiation of
\eqref{eq:draft-h-function} gives
\begin{equation}\label{eq:draft-h-derivative}
	h_v'(z)
	=
	\frac{z^2}{\sqrt{1-z^2}}
	\left[
	B_v\bigl(1+2F(z)\bigr)-2
	\right],
	\qquad
	F(z)=\frac{\sqrt{1-z^2}\,\arcsin z}{z}.
\end{equation}
The function $F$ is decreasing on $(0,1)$.  Indeed, writing
$z=\sin\theta$ gives
\[
F(z)=\theta\cot\theta,
\]
and
\[
\frac{d}{d\theta}(\theta\cot\theta)
=
\frac{\sin\theta\cos\theta-\theta}{\sin^2\theta}<0.
\]
Since $z\leq q$, we therefore have
\[
F(z)\geq F(q).
\]
Consequently, by \eqref{eq:draft-h-derivative}, it is enough to prove
\begin{equation}\label{eq:draft-B-condition}
	B_v\bigl(1+2F(q)\bigr)\geq2.
\end{equation}

Set
\[
x=q^2\in\left(0,\frac13\right],
\qquad
a=\frac{v}{q}=\frac{\arctan q}{q}.
\]
Then
\[
B_v=\frac{1+(1+x)a}{a+2}.
\]
The right-hand side is increasing as a function of $a$, since
\[
 \frac{d}{da}\frac{1+(1+x)a}{a+2}
 =\frac{1+2x}{(a+2)^2}>0,
\]
and
\[
a
=\int_0^1\frac{ds}{1+xs^2}
\geq1-\frac{x}{3}.
\]
Hence
\begin{equation}\label{eq:draft-B-lower}
	B_v
	\geq
	\frac{6+2x-x^2}{9-x}.
\end{equation}
On the other hand,
\[
\frac{\arcsin q}{q}
=
\int_0^1\frac{ds}{\sqrt{1-xs^2}}
\geq1+\frac{x}{6},
\]
and therefore
\[
F(q)
=
\sqrt{1-x}\,\frac{\arcsin q}{q}
\geq
\sqrt{1-x}\left(1+\frac{x}{6}\right).
\]
For $0\leq x\leq1/3$,
\[
(1-x)\left(1+\frac{x}{6}\right)^2
-
\left(1-\frac{x}{2}\right)^2
=
\frac{x(12-20x-x^2)}{36}
\geq0.
\]
Both sides being nonnegative, it follows that
\begin{equation}\label{eq:draft-F-lower}
	F(q)\geq1-\frac{x}{2}.
\end{equation}
Combining \eqref{eq:draft-B-lower} and \eqref{eq:draft-F-lower}, we obtain
\begin{align*}
	B_v\bigl(1+2F(q)\bigr)
	&\geq
	\frac{6+2x-x^2}{9-x}(3-x),\\
	\frac{6+2x-x^2}{9-x}(3-x)-2
	&=
	\frac{x(x^2-5x+2)}{9-x}.
\end{align*}
Since $0<x\leq1/3$ and $2x-5<0$ on this interval, we have
\[
x^2-5x+2
\geq
\frac19-\frac53+2
=
\frac49>0,
\]
condition \eqref{eq:draft-B-condition} follows.  Thus
\eqref{eq:draft-h-derivative} gives $h_v'(z)\geq0$, and since
$h_v(0)=0$, we have $h_v(z)\geq0$.  By
\eqref{eq:draft-g-second-derivative},
\[
g_\alpha''(t)\geq0,
\]
so $g_\alpha$ is convex on the active range.

We now apply this convexity to the splitting.  Let $\tau>0$ be
determined by
\begin{equation}\label{eq:draft-tau}
	\psi(\tau)=\alpha.
\end{equation}
Since both $\beta_i=\psi(t_i)$ are larger than $\alpha$, we have
$t_i>\tau$, and hence
\begin{equation}\label{eq:draft-T-two-tau}
	T>2\tau.
\end{equation}
On the closed two-sided active interval $[\tau,T-\tau]$, define
\begin{equation}\label{eq:draft-J}
	J(s)
	=
	g_\alpha(s)+g_\alpha(T-s)
	=
	f(\alpha,\psi(s))+f(\alpha,\psi(T-s)).
\end{equation}
Both arguments of $g_\alpha$ lie in its active range on this interval.
Since $g_\alpha$ is convex, so is $J$.  Hence the maximum of $J$ on
$[\tau,T-\tau]$ is attained at an endpoint.  The two endpoint values agree
by the symmetry $J(s)=J(T-s)$.  For the given splitting,
$t_1,t_2>\tau$, and the longest-edge rule gives
\[
H(\beta_1,\alpha)+H(\beta_2,\alpha)
=J(t_1)
\leq J(\tau).
\]
Using \eqref{eq:draft-tau}, the definition of $D_\alpha$, and the
additivity identity \eqref{eq:draft-additivity}, we find
\begin{align*}
	J(\tau)
	&=
	f(\alpha,\psi(\tau))
	+f(\alpha,\psi(T-\tau))\\
	&=
	f(\psi(\tau),\alpha)
	+f(\psi(T-\tau),\alpha)
	+D_\alpha(T-\tau)\\
	&=
	f(\beta,\alpha)+D_\alpha(T-\tau).
\end{align*}
By \eqref{eq:draft-T-two-tau},
\[
\tau<T-\tau<T.
\]
Strict monotonicity of $D_\alpha$ on the active range therefore yields
\[
J(\tau)
=
f(\beta,\alpha)+D_\alpha(T-\tau)
<
f(\beta,\alpha)+D_\alpha(T)
=
f(\alpha,\beta)
=
H(\beta,\alpha).
\]
Consequently,
\[
H(\beta_1,\alpha)+H(\beta_2,\alpha)
<
H(\beta,\alpha)
=
f(\alpha,\beta).
\]

In all three cases,
\[
H(\beta_1,\alpha)+H(\beta_2,\alpha)
<
H(\beta,\alpha)
=
f(\alpha,\beta).
\]
Substitution into
\eqref{eq:draft-area-Y}--\eqref{eq:draft-area-X}
proves \eqref{eq:draft-case-long-difference} and therefore
\eqref{eq:draft-case-long-strict}.
\end{proof}

Combining Propositions~\ref{prop:draft-beta-at-most-alpha}--
\ref{prop:draft-beta-greater-alpha} gives
\begin{equation}\label{eq:draft-final-deletion-comparison}
 \boxed{\Amin(X)\geq\Amin(X\setminus\{x\}),}
\end{equation}
with equality exactly when $\theta(yz)\leq\theta(ab)$.  Since every body
involved has constant width one, Blaschke's relation
\eqref{eq:blaschke-relation} gives the same comparison and equality
characterization for volume.

\begin{corollary}[dangling-free reduction]
Successively deleting dangling points from an extremal generating set
produces an extremal generating set without dangling points and with no
larger optimal area.  Consequently, if the minimum over all Meissner
polyhedra generated by extremal finite sets is attained, it is attained by
one having a generating set without dangling points.
\end{corollary}

\begin{proof}
Each deletion preserves extremality and, by
\eqref{eq:draft-final-deletion-comparison}, does not increase $\Amin$.
The process terminates because the number of points decreases at every
step.  An extremal set with four points has all
$2\cdot4-2=\binom42$ diametric pairs, so it has no dangling point.  Thus
the terminal set is dangling-free.  If the initial set realizes the global
minimum, the terminal set cannot have smaller area and therefore realizes
the same minimum.
\end{proof}

\section{Conclusions and perspectives}

The main result of this paper shows that inserting dangling points into an
extremal generating set cannot decrease the minimum area obtainable by
Meissner smoothing.  This treats one of the simplest natural enrichments of
an extremal generating set.  We state some natural generalizations that may
be useful in studying the conjectured volume minimality of the Meissner
tetrahedra.

The first conjecture concerns arbitrary inclusions of extremal finite sets. Compared to the results of this paper, where a single dangling point is added to an existing extremal set, adding multiple points may yield more complex combinatorics of diameter graphs. Excluding such possibilities requires different approaches than the one presented in this paper. 

\begin{conj}\label{conj:monotone}
	Let $Y\subsetneq X$ be two extremal finite sets of unit diameter. Then
	\[
	\Amin(Y)\leq\Amin(X).
	\]
	The optimal Meissner area is nondecreasing with respect to inclusion
	of the generating set.
\end{conj}

The density of Meissner polyhedra among bodies of constant width motivates
the following constrained minimization problem.  For an extremal finite set
$X$ of unit diameter, denote by $\mathcal{CW}_1(X)$ the class of convex
bodies of constant width one that contain $X$.

\begin{conj}\label{conj:containMeissner}
	The minimum surface area over $\mathcal{CW}_1(X)$ is attained by a
	Meissner polyhedron based on $X$.
\end{conj}
By Blaschke's relation, the corresponding volume statement is equivalent.
The following special case is also open.
\begin{conj}\label{conj:containTetra}
	Let $T$ be the vertex set of a regular tetrahedron of side length one.
	Then the minimum surface area over $\mathcal{CW}_1(T)$ is attained by a
	Meissner tetrahedron.
\end{conj}

%\noindent{\bf Acknowledgements:} The author was partially supported by the ANR Project: STOIQUES.

\bibliographystyle{abbrv}
\bibliography{biblio}

\noindent {\bf Beniamin Bogosel}

\noindent Faculty of Exact Sciences, Aurel Vlaicu University of Arad\\
 2 Elena Dr\u agoi Street, Arad, Romania
 
\noindent Email address: \texttt{beniamin.bogosel@uav.ro}

\end{document}